\documentclass[runningheads,a4paper]{llncs}

\usepackage{amsmath,amssymb}
\usepackage{enumerate}

\usepackage{hyperref}

\usepackage{cleveref}

\def\N{\mathbb N}
\def\A{\mathcal A}
\def\B{\mathcal B}
\def\C{\mathcal C}

\def\PT{{\mathcal P}_{\Theta}}
\def\DT{{D}_{\Theta}}
\def\DTa{{D}_{\Theta_1}}

\def\Lan{ \mathcal{L}}
\def\Lu{{\mathcal L}(\uu)}
\def\uu{\mathbf u}
\def\vv{\mathbf v}

\def\Pal{\rm Pal}
\def\PalT{{\rm Pal}_{\Theta}}

\def \Rk#1 {$\mathcal{R}_{#1}$}

\def \FC#1 {
\mathcal{C}
\ifthenelse{\equal{#1}{}}{}{(#1)}
}

\def \PC#1 {
\mathcal{P}_{\Theta}
\ifthenelse{\equal{#1}{}}{}{(#1)}
}

\def \PCn#1 {
\mathcal{P}
\ifthenelse{\equal{#1}{}}{}{(#1)}
}

\def \TuT {T_{\Theta}}

\def \Tr {\Theta_0}
\def \Ta {\Theta_1}
\def \Tb {\Theta_2}
\def \gT {\gamma_{\Theta}}

\spnewtheorem{thm}{Theorem}{\bfseries}{\itshape}
\spnewtheorem{prop}[thm]{Proposition}{\bfseries}{\itshape}
\spnewtheorem{lem}[thm]{Lemma}{\bfseries}{\itshape}
\spnewtheorem{coro}[thm]{Corollary}{\bfseries}{\itshape}
\spnewtheorem{defi}[thm]{Definition}{\bfseries}{\itshape}
\spnewtheorem{exam}[thm]{Example}{\itshape}{\itshape}
\spnewtheorem{rem}[thm]{Remark}{\itshape}{\rmfamily}

\crefname{thm}{theorem}{theorems}
\crefname{coro}{corollary}{corollaries}
\crefname{exam}{example}{examples}
\crefname{lem}{lemma}{lemmas}
\crefname{prop}{proposition}{propositions}
\crefname{defi}{definition}{definitions}

\newcommand{\keywords}[1]{\par\addvspace\baselineskip
\noindent\keywordname\enspace\ignorespaces#1}

\begin{document}

\mainmatter  

\title{Infinite words rich and almost rich \\ in generalized palindromes}

\titlerunning{Infinite words rich and almost rich in generalized palindromes}

%
%

\author{Edita Pelantov\'a \and \v St\v ep\'an Starosta}

\authorrunning{Edita Pelantov\'a \and \v St\v ep\'an Starosta}

\institute{Department of Mathematics, FNSPE, Czech Technical University in Prague, Trojanova 13, 120~00 Praha~2, Czech Republic}

%
%

\maketitle

\begin{abstract}

We focus on $\Theta$-rich and almost $\Theta$-rich words over a
finite alphabet $\mathcal{A}$, where $\Theta$ is an involutive
antimorphism over $\mathcal{A}^*$. We show that any recurrent
almost $\Theta$-rich word $\uu$ is an image of a recurrent
$\Theta'$-rich word under a suitable morphism, where $\Theta'$ is again an involutive antimorphism. Moreover, if the word $\uu$
is  uniformly recurrent, we show that $\Theta'$ can be set to the reversal mapping.
We also treat one special case of almost $\Theta$-rich words.
We show that every $\Theta$-standard words with seed is an image of an Arnoux-Rauzy word.
\keywords{palindrome, palindromic defect, richness}
\end{abstract}

\section{Introduction}

In this paper we will deal with infinite words over a finite
alphabet $\A$. A word ${\bf u}\in \mathcal{A}^\mathbb{N}$ we are
interested in has its language $\mathcal{L}({\bf u})$ saturated,
in a certain sense, by generalized palindromes,  here called
$\Theta$-palindromes.  We will use the symbol $\Theta$ for an
involutive antimorphism, i.e., a mapping  $\Theta: \A^* \mapsto
\A^*$ such that $\Theta^2 = \rm{Id}$ and $\Theta(uv) =
\Theta(v)\Theta(u)$ for all $u,v \in \A^*$. Fixed points of
$\Theta$ are called $\Theta$-palindromes. A word $w$ is a
$\Theta$-palindrome if $\Theta(w) = w$. The most common
antimorphism used in combinatorics on words is the reversal
mapping. We will denote it by $\Tr$. The reversal mapping
associates to every word $w =  w_1 w_2 \ldots w_n$ its mirror
image $\Tr(w) = w_n w_{n-1} \ldots w_1$. In the case $w = \Tr(w)$,
we will sometimes say that $w$ is a palindrome or classical
palindrome instead of $\Tr$-palindrome.

The set of distinct $\Theta$-palindromes occurring in a finite
word $w$ will be denoted $\PalT(w)$. Since the empty word
$\varepsilon$ is a $\Theta$-palindrome for any $\Theta$, we have
a simple lower  bound $\#\PalT(w) \geq 1$.

In 2001, Droubay et al. gave in \cite{DrJuPi} an upper bound for
the reversal mapping $\Tr$. They deduced that  $\# \Pal_{\Tr} (w)
\leq |w| + 1$, where $|w|$ denotes the length of the word $w$. In
\cite{BlBrGaLa}, Brlek et al. studied involutive antimorphisms
with no fixed points of length $1$.  For such $\Theta$ they diminished the upper
bound,  they showed that $\# \PalT(w) \leq |w|$ for all non-empty
word $w$. In \cite{Sta2010}, the upper bound is precised.
  The following estimate is valid for any involutive
antimorphism $\Theta$:
\begin{equation}\label{eq:horni_mez_poctu_pali}
\# \PalT(w) \leq |w| + 1 - \gT(w),
\end{equation}
where $\gT(w) := \# \left \{ \{ a, \Theta(a)\} \mid a \text{ occurs in } w \text{ and } a \neq \Theta(a) \right \}$.
Let us note that if $\Theta = \Tr$, then $\gT(w) = 0$, and the upper bound in \eqref{eq:horni_mez_poctu_pali}
is the same bound as for usual palindromes.

According to the terminology for classical palindromes introduced
in \cite{GlJuWiZa} and for $\Theta$-palindromes in \cite{Sta2010}, we will say that a finite word $w$ is {\em
$\Theta$-rich} if  the equality in \eqref{eq:horni_mez_poctu_pali} holds. An infinite word $\uu \in
\A^{\N}$ is {\em $\Theta$-rich} if any its factor $w \in \Lu$ is
$\Theta$-rich. In \cite{BrHaNiRe}, the authors introduced the
\textit{palindromic defect} of a finite word $w$ as the difference
between the upper bound $|w| + 1$ and the actual number of
palindromic factors. We define analogously the
\textit{$\Theta$-palindromic defect} of $w$ as
$$
\DT(w) : = |w| + 1 - \gT(w) - \# \PalT(w).
$$
We define for an infinite word $\uu$ its
\textit{$\Theta$-palindromic defect} as
$$
\DT(\uu) = \sup \{ \DT(w) \mid w \in \Lu \}.
$$
Words with finite $\Theta$-palindromic defect will be referred to
as {\em almost $\Theta$-rich}.

In \cite{BuLuGlZa}, it is shown that  rich words (i.e. $\Tr$-rich
words) can be characterized using an inequality shown in
\cite{BaMaPe} for infinite words with languages closed under
reversal. Results of both mentioned papers were generalized in
\cite{Sta2010} for an arbitrary involutive antimorphism. In
particular, it is shown that if an infinite word has its language
closed under $\Theta$, the following inequality holds
\begin{equation} \label{eq:nerovnost}
 \C (n+1) -  \C (n) + 2 \geq \PT(n) + \PT(n+1) \text{   for all } n
\geq 1,
\end{equation}
where $\C(n)$ is the \textit{factor complexity} defined by $\C(n) := \# \{
w \in \Lu \mid  n =|w|  \}$ and $\PT(n)$ is the \textit{$\Theta$-palindromic
complexity} defined by $\PT(n) := \# \{ w \in \Lu \mid w =
\Theta(w) \text{ and } n =|w| \}$. The gap between the left-hand side and
 the right-hand side in \eqref{eq:nerovnost} decides about
$\Theta$-richness.   Let us therefore  denote by $\TuT(n)$
the quantity
$$
\TuT(n) := \C(n+1) - \C(u) + 2 - \PT(n+1) - \PT(n).
$$
In \cite{Sta2010}, it is also shown that an infinite word with
language closed under $\Theta$ is $\Theta$-rich if and only if
$$
\TuT(n) = 0 \text{   for all } n \geq 1.
$$

The list of infinite words which are $\Tr$-rich is quite
extensive. See for instance \cite{BaMaPe,BuLuLu,BuLuLuZa}.
 Examples of $\Theta$-rich words can be found in
\cite{AnZaZo}.
  Fewer examples of words with finite non-zero palindromic defect are
known. Periodic words with finite non-zero $\Tr$-defect can be
found in \cite{BrHaNiRe}, aperiodic ones  are studied in
\cite{GlJuWiZa} and \cite{BaPeSta3}. To our knowledge, examples of
words with $0 < \DT(\uu) < +\infty$ and $\Theta \neq \Tr$ have not
yet been explicitly exhibited. As we will show, such examples are
$\Theta$-standard words with seed defined in \cite{BuLuLuZa2}
and thus also their subset, standard $\Theta$-episturmian words, which can be constructed from standard episturmian words
(see \cite{BuLuLu_cha}).

 The main aim of this
paper is to show that among words with finite
$\Theta$-pa\-lindromic defect, $\Theta$-rich words, i.e. words with
$\DT(\uu) = 0$, play an important role. We will show the following
theorems.

\begin{thm} \label{thm:1}
Let $\Ta: \A^* \mapsto \A^*$ be an involutive antimorphism.
Let $\uu \in \A^{\N}$ be an infinite recurrent word such that
$\DTa(\uu) < +\infty$. Then there exist an involutive antimorphism
$\Tb: \B^* \mapsto \B^*$,  a morphism $\varphi: \B^* \mapsto \A^*$
and an infinite recurrent word $\vv \in \B^{\N}$ such that
$$
\uu = \varphi(\vv) \text{ and } \vv \text{ is } \Tb\text{-rich}.
$$
\end{thm}
A stronger statement can be shown if  the requirement of uniform recurrence is imposed on the word $\uu$.

\begin{thm} \label{thm:2}
Let $\Theta: \A^* \mapsto \A^*$ be an involutive antimorphism.
Let $\uu \in \A^{\N}$ be an infinite uniformly recurrent word such that $\DT(\uu) < +\infty$. Then there exist a
morphism $\varphi: \B^* \mapsto \A^*$ and an infinite uniformly
recurrent  word $\vv \in \B^{\N}$ such that
$$
\uu = \varphi(\vv) \text{ and } \vv \text{ is } \Tr\text{-rich}.
$$
\end{thm}
One can conclude that rich words, using the classical
notion of palindrome, play somewhat more important role that
$\Theta$-rich words for an arbitrary $\Theta \neq \Tr$.

Proofs of the two stated theorems do not provide any relation
between the size of the alphabet $\B$ of the word $\vv$ and the
size of the original alphabet $\A$.  In the following special
case,  the size of $\B$ can be bounded.
Moreover, the word $\vv$ is more specific.

\begin{thm} \label{thm:3}
Let $\Theta: \A^* \mapsto \A^*$ be an involutive antimorphism
 and $\uu \in \A^{\N}$
be a $\Theta$-standard word with seed.
Then there exist an Arnoux-Rauzy word $\vv \in \B^{\N}$
and a morphism $\varphi: \B^* \mapsto \A^*$
such that
$$
\uu = \varphi(\vv) \text{ and } \# \B \leq \# \A.
$$
\end{thm}

All three mentioned theorems present almost $\Theta_1$-rich word
as an image of a $\Theta_2$-rich word by a suitable morphism. The
opposite question when a morphic image of a $\Theta_1$-rich word
is almost $\Theta_2$-rich is not tackled here. In \cite{GlJuWiZa},
a type of morphisms preserving the set of almost $\Theta_0$-rich
words is studied.

\section{Properties of words with finite $\Theta$-defect}

We will consider mainly infinite words $\uu  = (u_n)_{n \in \N}\in \A^{\N}$ having their language $\Lu$ closed under a
given involutive antimorphism $\Theta$. In other words, for any
factor $w \in \Lu$ we have $\Theta(w) \in \Lu$.

For any factor $w\in \mathcal{L}({\mathbf u})$ there exists an
index $i$ such that $w$ is a prefix  of  the infinite word
$u_iu_{i+1}u_{i+2} \ldots$. Such an index is called an {\em
occurrence} of $w$ in ${\mathbf u}$. If each factor of $\mathbf u$
has infinitely many occurrences in ${\mathbf u}$, the infinite
word $\mathbf u$ is said to be {\em recurrent}. It is easy to see
that if the language of ${\mathbf u}$ is closed under $\Theta$,
then ${\mathbf u}$ is recurrent. For a~recurrent infinite word
${\mathbf u}$, we may define the notion of a~{\em complete return
word} of any $w \in\mathcal{L}({\mathbf u})$. It is a~factor $v\in
\mathcal{L}({\mathbf u})$ such that $w$ is a prefix and a suffix
of $v$ and $w$ occurs in $v$ exactly twice. Under a~{\em return
word} of a factor $w$ we usually mean a word $q \in
\mathcal{L}({\mathbf u})$  such that $qw$ is a complete return
word of $w$. If any factor $w \in \mathcal{L}({\mathbf u})$ has
only finitely many return words, then the infinite word ${\mathbf
u}$ is called {\em uniformly recurrent}.

An important role for the description of languages closed under
$\Theta$ is played by the so-called super reduced Rauzy graphs
$G_n({\mathbf u})$. Before defining them,
we will introduce some necessary notions.

We say that a factor $w\in \Lu$ is left special (LS) if $w$ has at
least two left extensions, i.e., if  there exist two letters $a,b
\in \A$, $a \neq b$, such that $aw, bw \in \Lu$. A right special
(RS) factor is defined analogously. If a factor is LS and RS, we
refer to it as bispecial. The closedness under $\Theta$ assures
the following relation:  a factor $w$ is LS if and only if the
factor $\Theta(w)$ is RS.

An {\em $n$-simple path} $e$ is a~factor of ${\mathbf u}$ of
length at least $n + 1$ such that the only special (right or left)
factors of length $n$ occurring in $e$ are its prefix and suffix
of length $n$. If $w$ is the prefix of $e$ of length $n$ and $v$
is the suffix of $e$ of length $n$, we say that the $n$-simple
path $e$ begins with $w$ and ends with $v$. We will denote by
$G_n({\mathbf u})$ an undirected graph whose set of vertices is
formed by unordered pairs $(w,\Theta(w))$ such that $w \in
\mathcal{L}({\mathbf u})$, $|w| = n$, is RS or LS. We connect
two vertices $(w,\Theta(w))$ and $(v,\Theta(v))$ by an unordered
pair $(e,\Theta(e))$ if $e$ or $\Theta(e)$ is an $n$-simple path
beginning with $w$ or $\Theta(w)$ and ending with $v$ or $\Theta(v)$.
Note that the graph $G_n({\mathbf u})$ may have multiple edges and
loops.

Surprisingly, the super reduced Rauzy graph  $G_n(\uu)$ can be used to
detect the equality in \eqref{eq:nerovnost}. Let us cite Corollary
7 from \cite{Sta2010}.
\begin{prop} \label{theta_graf_rovnost}
Let $n \in \N$ and $\Lu$ be closed under $\Theta$.
Then $\TuT(n) = 0$ if and only if
\begin{enumerate}
  \item all $n$-simple paths forming a loop in $G_n(\uu)$ are $\Theta$-palindromes and
  \item $G_n(\uu)$  after removing loops is a tree.
\end{enumerate}
\end{prop}

Analogously to the case of the reversal mapping, one can see from the definition of $\Theta$-defect that an infinite word $\uu$ has finite $\Theta$-defect
if and only if there exists an integer $H$ such that 
of every prefix $p$ of $\uu$ of length greater than $H$ has a unioccurrent $\Theta$-palindromic suffix, i.e., a suffix occurring exactly once in $p$.
We will use this fact to prove the following lemma.

\begin{lem} \label{lem:R_a_kon_def_je_CuT}
Let $\uu$ be a recurrent infinite word with finite $\Theta$-defect.
Then $\Lu$ is closed under $\Theta$.
\end{lem}
\begin{proof}
Let $H$ be an integer such that every prefix of $\uu$ of length greater
 than $H$ has a unioccurrent $\Theta$-palindromic suffix.
Suppose that $w$ is a factor of $\uu$ such that $\Theta(w) \not
\in \Lu$. Since $\uu$ is recurrent, we can find two
consecutive occurrences $i$ and $j$ of the factor $w$ such that
$i, j > H$ and $i < j$. Denote $p$ the prefix of $\uu$ ending with
$w$ occurring at $j$, i.e., $|p| = j + |w|$. Since $|p|
> H$, there exists a unioccurrent $\Theta$-palindromic suffix of
$p$. Denote $s$ to be such a suffix. If $|s| \leq |w|$, then $s$
is a factor of $w$ and thus occurs at least twice in $p$
- a contradiction with the unioccurrence of $s$. If $|s| > |w|$,
the $w$ is a factor of $s$ which is a $\Theta$-palindrome and thus
contains $\Theta(w)$ as well - a contradiction with the
assumption that $\Theta(w) \not \in \Lu$.
\end{proof}

In \cite{BaPeSta3},  various properties are shown for words with
finite $\Tr$-palindromic defect.  These properties and their
proofs are valid  even if we replace the antimorphism
$\Theta_0$ by an arbitrary $\Theta$. Therefore, we mention here the
relevant statements without proving them.

\begin{prop} \label{finite_defect_CuT_properties}
Let $\uu$ be an infinite recurrent word such that $\DT(\uu) < +\infty$. Then there exists a positive  integer
$H$ such that
\begin{itemize}
  \item every prefix of $\uu$ longer than $H$ has a unioccurent $\Theta$-palindromic suffix;
  \item for any factor $ w \in \Lu$ such that $|w| > H$, occurrences of $w$ and $\Theta(w)$ in the word  $\uu$ alternate;
  \item for  any $w \in \Lu$ such that $|w| > H$,  every factor $v \in \Lu$ beginning with $w$, ending with $\Theta(w)$, and with no other occurrences of $w$ or $\Theta(w)$ is a
  $\Theta$-palindrome;
  \item  $\TuT(n) = 0$ for any integer $n > H$.
\end{itemize}
\end{prop}
As already mentioned,  the first property
listed in the previous proposition, in fact, characterizes
words with finite $\Theta$-defect. We do not know whether this
is the case of the remaining properties. If we restrict our
attention to uniformly recurrent words, only then several
characterizations of words with finite $\Theta$-defect can be
shown. The next proposition states two of them that we will use in
what follows. Again, the proposition  is based on the work done in
\cite{BaPeSta3} for $\Theta = \Tr$. No modifications besides
replacing $\Theta_0$ by $\Theta$ in its proof are needed,
therefore, we will omit it.

\begin{prop} \label{finite_defect_CuT_characterizations}
Let $\uu$ be a uniformly recurrent infinite word  with language
closed under $\Theta$. The following statements are equivalent.
\begin{itemize}
  \item $\DT(\uu) < + \infty$;
  \item \label{finite_defect_CuT_characterizations_2} there exists a positive integer  $ K$ such that for  any  $\Theta$-palindrome $w \in \Lu$
  of length  $|w| \geq K$,   all complete return words of $w$ are
  $\Theta$-palindromes;
  \item there exists a positive integer  $ H$ such that for any $w \in \Lu$, the longest $\Theta$-palindromic suffix of $w$  is unioccurrent in $w$.
\end{itemize}
\end{prop}

A $\Theta$-standard word   with seed is an infinite word defined
by using $\Theta$-palindromic closure, for details see
\cite{BuLuLuZa2}. Construction of such word $\uu$ guarantees  that
$\uu$ is uniformly recurrent (cf. Proposition 3.5. in \cite{BuLuLuZa2}).  The authors of \cite{BuLuLuZa2}
showed (Proposition 4.8) that any complete return word of a
sufficiently long  $\Theta$-palindromic factor is a
$\Theta$-palindrome as well. Therefore, $\Theta$-standard words
with seed serve as an example of almost $\Theta$-rich words.

\begin{coro} \label{co:std_seed_are_rich}
Let $\uu$ be a $\Theta$-standard word with seed.
Then $\DT(\uu) < +\infty$.
\end{coro}

\section{Proofs}

In this section we give proofs of all three theorems stated in
Introduction. Although \Cref{thm:2} seems to be only a
refinement of \Cref{thm:1}, constructions of the
morphisms  $\varphi$ in their proofs differ substantially. It is
caused by stronger properties  we may exploit for a uniformly
recurrent word.

\begin{proof}[Proof of \Cref{thm:1}]
Recall that according to \Cref{lem:R_a_kon_def_je_CuT} the language $\Lu$ is closed under $\Ta$.

 If $\uu$ is an eventually periodic word with language closed under $\Theta_1$, then $\uu$ is purely periodic.  Any   purely periodic word
  is a morphic image of a word ${\bf v}$ over one-letter alphabet under the morphism which assigns
  to this letter the period of $\uu$. Therefore we may assume
  without loss of generality that $\uu$ is not eventually
  periodic.

  Since $D_{\Theta_1}(\uu) < +
\infty$, according to
\Cref{finite_defect_CuT_properties,theta_graf_rovnost}, there
exists $H \in \N$ such that
\begin{enumerate}
  \item \label{proof_prop_1} $\forall w \in \Lu$, $|w| > H$, occurrences of $w$ and $\Theta_1(w)$ alternate;
  \item \label{proof_prop_2} $\forall w \in \Lu$, $|w| > H$, every factor beginning with $w$, ending with $\Theta_1(w)$ and with
  no other occurrences of $w$ or $\Theta_1(w)$ is a $\Theta_1$-palindrome;
  \item \label{proof_prop_3} $\forall n \geq H$, every loop in $G_n(\uu)$ is a $\Theta_1$-palindrome and $G_n(\uu)$ after removing loops is a tree.
\end{enumerate}

Fix $n > H$.
If an edge $(b, \Theta_1(b)) $ in $G_n(\uu)$ is a loop,
then, according to the property \ref{proof_prop_3}, we have $b =
\Theta_1(b)$. If the edge $(b, \Theta_1(b)) $ connects  two
distinct vertices $(w_1,\Theta_1(w_1))$ and $(w_2,\Theta_1(w_2))$,
then there exist exactly two  $n$-simple paths $b$  and $ \Theta_1
(b)$ such that WLOG the $n$-simple path $b$ begins with $w_1$ and
ends with $w_2$ and the simple path $ \Theta_1 (b)$ begins with
$\Theta_1(w_2)$ and ends with $\Theta_1(w_1)$.

We assign to every $n$-simple path $b$ a new symbol $[b]$, i.e., we define the alphabet $\B$ as
$$
\B := \left \{ [b] \mid b \in \Lu \text{ is an } n\text{-simple
path} \right \}
$$
and on this alphabet we define an involutive antimorphism
$\Theta_2: \B^* \mapsto \B^*$ in the following way:
$$
\Theta_2([b]) := [\Theta_1(b)].
$$

We are now going to construct a suitable infinite word $\vv \in
\B^{\N}$. Let $(s_i)_{i \in \N}$ denote a strictly increasing
sequence of indices such that $s_i$ is an occurrence of RS or LS
factor of length $n$ and  every RS and LS factor of length $n$
occurs at some index $s_i$. We define $\vv = (v_i)_{i \in \N}$ by
the formula
$$
v_i = [b] \quad \text{   if  } \quad b = u_{s_i}u_{s_{i}+1} \ldots u_{s_{i+1}+n-1}.
$$

This construction can be done for any $n > H$. Since infinitely
many prefixes of $\uu$ are LS or RS factors, we can choose such $n
> H$ that the prefix of $\uu$ of length $n$ is LS or RS, i.e.,
$s_0 = 0$.

According to Proposition 12 in \cite{Sta2010}, to prove that $\vv$
is $\Theta_2$-rich we need to show the following:
\begin{enumerate}[(i)]
  \item \label{proof_rich_1} for every non-empty factor $ w \in \Lan(\vv)$, any factor $v$ beginning with $w$ and ending with $\Theta_2(w)$,
   with no other occurrences of $w$ or $\Theta_2(w)$, is a $\Theta_2$-palindrome;
  \item \label{proof_rich_2} for every letter $[b] \in \B$ such that $[b]\neq \Theta_2([b])$, the occurrences of $[b]$ and $\Theta_2([b])$
  in  the word $\vv$ alternate.
\end{enumerate}

Let us first verify \eqref{proof_rich_1}. Let $e$ and $f$ be
factors of $\vv$ such that $e$ is a prefix and $\Theta_2(e)$ is a
suffix of $f$ and there are no other occurrences of $e$ or
$\Theta_2(e)$ in $f$. In that case  there exist integers  $r \leq
k$ such that $f = [b_1][b_2] \ldots [b_k]$ and   $e = [b_1][b_2]
\ldots [b_r]$. The case $r = k$ is trivial. Suppose $r < k$. Since
$\vv$ is defined as a coding of consecutive occurrences of
$n$-simple paths in $\uu$, factor $f$ codes a certain segment of
the word $\uu$. Let us denote that segment $F = u_j \ldots u_l$
where $j = s_t$ for some $t \in \N$ and $l = s_{t + k - 1} + n
-1$. Factor $e$ codes in the same way a factor $E = u_j \ldots
u_h$ where $h = s_{t+r-1}+n-1$.

Due to the definition of $\Theta_2$, the fact that  $e$ is a prefix of
$f$ and $\Theta_2(e)$ is a suffix of $f$ ensures that $E$ is a
prefix of $F$ and $\Theta(E)$ is a suffix of $F$. Suppose $f$ is
not a $\Theta_2$-palindrome. This implies that $F$ is not a
$\Theta_1$-palindrome which contradicts the property
\ref{proof_prop_3}.

Let us now verify \eqref{proof_rich_2}. Consider $[b] \in
\B$ such that $[b]\neq \Theta_2([b])$. Moving along the infinite
word  $\uu= u_0u_1u_2 \ldots$ from  the left to the right with a
window  of width  $n$ corresponds to a walk in the graph $G_n(\uu)$.
The pair $b$ and $\Theta_1(b)$  of $n$-simple paths in $\uu$
represents an edge in $G_n(\uu)$ connecting two distinct vertices.
Moreover, moving along the $n$-simple paths $b$ and moving along
$\Theta_1(b)$ can be viewed as traversing that edge in opposite
directions. Since $G_n(\uu)$ after removing loops is a tree, the only
way to traverse an edge is alternately in one direction and in the
other. Thus, the occurrences of letters $[b]$ and $ \Theta_2([b])$
in $\vv$ alternate.

We have shown that $\vv$ is $\Theta_2$-rich. It is now obvious how
to define a morphism $\varphi: \B^* \mapsto \A^*$. If an
$n$-simple path $b$ equals $b = u_{s_i} u_{s_i + 1} \ldots
u_{s_{i+1} + n - 1} $, then we set $\varphi([b]) := u_{s_i} u_{s_i
+ 1} \ldots u_{s_{i+1} - 1}$.

\end{proof}

\begin{proof}[Proof of \Cref{thm:2}]
Recall again that according to \Cref{lem:R_a_kon_def_je_CuT} the language $\Lu$ is closed under $\Theta$.

Next, we show that infinitely many $\Theta$-palindromes are also prefixes of $\uu$.
 Consider an integer $H$
whose existence is guaranteed by Proposition
\ref{finite_defect_CuT_properties}  and denote by $w$  a prefix of
$\uu$  longer than $H$. Since occurrences of factors $w$ and
$\Theta(w)$ in $\uu$ alternate,  according to the same proposition,
the prefix of $\uu$ ending with the first occurrence of
$\Theta(w)$ is a $\Theta$-palindrome.

Let us denote by $p$ a  $\Theta$-palindromic prefix of $\uu$ with
length $|p| > K$ where $K$ is the constant from Proposition
\ref{finite_defect_CuT_characterizations}. All complete return
words of $p$ are $\Theta$-palindromes.  Since $\uu$ is uniformly
recurrent, there exist only finite number of  complete return
words to $p$. Let  $r^{(1)}, r^{(2)}, \ldots, r^{(M)}$ be the list
of all these return words. Any complete return word $ r^{(i)}$ has
the form  $q^{(i)}p = r^{(i)}$ for some factor $q^{(i)}$, usually
called return word of $p$. Since $r^{(i)}$ and  $p$  are
$\Theta$-palindromes, we have
\begin{equation}\label{eq:zamena1}
p\Theta(q^{(i)}) = q^{(i)}p \ \hbox{\ for any return word } \
q^{(i)}.
\end{equation}

 Let us define a new alphabet
 $\mathcal{B} = \{ 1,2,\ldots, M\}$ and morphism $\varphi:
 \mathcal{B}^* \to  \mathcal{A}^*$ by the prescription
 $$\varphi(i)= q^{(i)},  \quad \hbox{for} \ \   i = 1,2,\ldots, M\,.$$
First we will check the validity of the relation
\begin{equation}\label{eq:zamena2}
\Theta\bigl(\varphi(w)p\bigr)= \varphi\bigl(\Theta_0(w)\bigr)p \ \
\  \hbox{for any }  \ w \in \mathcal{B}^*\,.
\end{equation}
Let $w = i_1i_2 \ldots i_n$. Then $
\Theta\bigl(\varphi(i_1i_2\ldots i_n)p\bigr)$ equals to
$$
\Theta(p)\Theta\bigl(\varphi(i_n)\bigr)\Theta\bigl(\varphi(i_{n-1})\bigr)\ldots\Theta\bigl(\varphi(i_1)\bigr)=
p\Theta\bigl(q^{(i_n)}\bigr)\Theta\bigl(q^{(i_{n-1})}\bigr)\ldots\Theta\bigl(q^{(i_1)}\bigr)
$$
and we may apply   gradually $n$ times  the equality
\eqref{eq:zamena1} to rewrite the right-hand side as
$$
q^{(i_n)}q^{(i_{n-1})}\ldots q^{(i_1)} p =
\varphi(i_n)\varphi(i_{n-1})\ldots \varphi({i_1})p =
\varphi\bigl(\Theta_0(i_1i_2\ldots i_n)\bigr)p.
$$
This proves the relation \eqref{eq:zamena2}.

An important property  of the morphism $\varphi$ is its
injectivity. Indeed,  in accordance with the definition,  number
of occurrences of the factor $p$ in $\varphi(w)p$ equals to the
number of letters in $w$ plus one.  Moreover,  each occurrence of
$p$ in $\varphi(w)p$ indicates beginning of an image of a letter
under $\varphi$. Therefore $\varphi(w)p= \varphi(v)p$  necessarily
implies  $w=v$.

Let us finally define the word $\vv$. As $p$ is a prefix of $\uu$,
the word $\uu$ can be written as a concatenation of return words
$q^{(i)}$ and thus we can determine a sequence $\vv =(v_n) \in
\mathcal{B}^\mathbb{N}$ such that
$$\uu =q^{(v_0)}q^{(v_1)}q^{(v_2)}\ldots $$
Directly from the definition of $\vv$ we have  $\uu =
\varphi(\vv)$. Since  $\uu$  is uniformly recurrent, the word
$\vv$ is uniformly recurrent as well. To prove that $\vv$ is a
$\Theta_0$-rich word, we will  show  that any complete return word of any
$\Theta_0$-palindrome in the word ${\mathbf v}$ is
a~$\Theta_0$-palindrome as well. According to Theorem 2.14  in \cite{GlJuWiZa},
this implies the $\Theta_0$-richness of ${\mathbf v}$.

Let $s$ be a~$\Theta_0$-palindrome in ${\mathbf v}$ and $w$ its
complete return word. Then $\varphi(w)p$ has precisely two
occurrences of the factor $\varphi(s)p$. Since  $s$ is
a~$\Theta_0$-palindrome, we have according to \eqref{eq:zamena2} that
$\varphi(s)p$ is a~$\Theta$-palindrome of length $|\varphi(s)p|
\geq |p|
> K$. Therefore $\varphi(w)p$ is a~complete return word of a long enough $\Theta$-palindrome and according to our assumption
$\varphi(w)p$ is a~$\Theta$-palindrome as well. Therefore by using
 \eqref{eq:zamena2} we have  $$\varphi(w)p =
\Theta\bigl(\varphi(w)p\bigr)= \varphi\bigl(\Theta_0(w)\bigr)p$$
and injectivity of $\varphi$ gives $w=\Theta_0(w)$, as we claimed.

\end{proof}

Theorem 6.1 in \cite{BuLuLu_cha} states that every standard $\Theta$-episturmian word is an image of a standard episturmian word.
Again, the role of $\Tr$ can be perceived as more important.
Also, compared to \Cref{thm:2}, it may be seen as a special case since $\Theta$-episturmian words, according to \Cref{co:std_seed_are_rich},
have finite $\Theta$-defect.

\begin{proof}[Proof of \Cref{thm:3}]

If $\uu$ is periodic, then the claim is trivial.
Suppose $\uu$ is aperiodic.

We are going repeat the proof of \Cref{thm:2} with a more specific
choice of $p$. Theorem 4.4 in \cite{BuLuLuZa2} implies that  there
exists $L \in \N$ such that any LS factor of $\uu$ longer than $L$
is a prefix of $\uu$. Without loss of generality, we may assume
that the constant $L$ is already chosen in such a way that all
prefixes of $\uu$ longer than $L$ have the same left extensions.
Let us denote their number by $M$. According to the same theorem,
infinitely many prefixes of $\uu$ are $\Theta$-palindromes and
thus bispecial factors as well.

According to Corollary \ref{co:std_seed_are_rich}, $\uu$ has finite
$\Theta$-palindromic defect.   Let $K$ be the constant from
\Cref{finite_defect_CuT_characterizations}. Altogether, there
exists a bispecial factor $p$, $|p|
> \max \{L,K\}$, such that it is a prefix of $\uu$ and a
$\Theta$-palindrome.  Since $p$ is longer than  $K$,  all complete
return words to  $p$ are $\Theta $-palindromes.  As $p$ is the
unique left special factor of length $|p|$ in $\uu$,  its return
words (i.e., complete return words after erasing the suffix $p$)
end with distinct letters. It means that there are exactly $M$
return words of $p$, denoted again $q^{(i)}$. Let us recall that
by $M$ we denoted  number of left extensions of some factor,
therefore $M\leq \# \A$.

The construction of the the word $\vv$ and prescription of the
morphism $\varphi$ over the alphabet $ \B = \{1,2,\ldots,M\}$ can
be done in exactly the same way as in the proof of \Cref{thm:2}.
It remains to show that $\vv$ is an Arnoux-Rauzy word.

According to \Cref{thm:2} we know that $\vv$ is $\Tr$-rich and uniformly recurrent.
Applying \Cref{lem:R_a_kon_def_je_CuT} we deduce that the language $\Lan(\vv)$ is closed under reversal.

Suppose there exist $v, w \in \Lan(\vv)$, two LS factors such that
$|v| = |w|$ and $v \neq w$. Since the words $q^{(i)}$ end with
distinct letters, it is clear that $\varphi(w)p$ is a LS factor of
$\uu$ and it has the same number of left extensions as $w$. The
same holds for $\varphi(v)p$. Since both these factors have their
length greater than or equal to $|p| > L$ and are both LS, one
must be prefix of another. Let WLOG $\varphi(w)p$ be a prefix of
$\varphi(v)p$, i.e., $\varphi(v)p = \varphi(ww')p$. The
injectivity of $\varphi$ implies $w' = \varepsilon$ and thus $v =
w$ -- a contradiction.
\end{proof}

\begin{rem}\Cref{thm:3} can be seen as a generalization of Theorem 6.1
in \cite{BuLuLu_cha} to $\Theta$-standard words with seed.
\end{rem}
\begin{rem}
Note also that the proof of \Cref{thm:3} is in fact a combination
of methods used in preceding proofs of \Cref{thm:1,thm:2} in the
sense that the  set of complete return words $r^{(i)}$  of the
factor $p$  and the set of  $|p|$-simple paths  in $\uu$ coincide.
\end{rem}

\section{Acknowledgement}

We acknowledge financial support by the Czech Science Foundation
grant GA\v CR 201/09/0584, by the grants MSM6840770039 and LC06002
of the Ministry of Education, Youth, and Sports of the Czech
Republic, and by the grant of the Grant Agency of the Czech Technical University in Prague.

%
%


\end{document}